\documentclass{article}
\usepackage[
  journal=JNCG,
  lang=british,
]{ems-journal}

\makeatletter
\renewcommand*\ps@titlepage{%
  \let\@oddfoot\@empty
  \let\@evenfoot\@empty
  \let\@oddhead\@empty
  \let\@evenhead\@empty
}
\makeatother

\usepackage{tikz-cd}
\usetikzlibrary{matrix}

\theoremstyle{plain}

\newtheorem{theorem}[subsubsection]{Theorem}
\newtheorem{conj}[subsubsection]{Conjecture}

\newtheorem{lemma}[subsubsection]{Lemma}
\newtheorem{proposition}[subsubsection]{Proposition}
\newtheorem{corollary}[subsubsection]{Corollary}

\theoremstyle{definition}

\newtheorem{definition}[subsubsection]{Definition}

\newtheorem{AI disclosure}[subsubsection]{AI disclosure}

\theoremstyle{remark}

\numberwithin{equation}{section}

\begin{document}

\title{The Bondal–Orlov Localization Conjecture Holds for Threefolds}

\emsauthor{3}{
	\givenname{Yu}
	\surname{Shen}
	\mrid{}
	\orcid{0000-0001-5766-1596}}{Y.~Shen}
\Emsaffil{3}{
  \department{Department of Mathematics}
  \organisation{Florida State University}
  \address{208 Love Building, 1017 Academic Way}
  \zip{32306-4510}
  \city{Tallahassee, FL}
  \country{USA}
  \affemail{ys26k@fsu.edu}
}

\emsauthor{4}{
	\givenname{Tianyang}
	\surname{Sun}
	\mrid{}
	\orcid{}}{T.~Sun}
\Emsaffil{4}{
  \department{School of Mathematical Sciences}
  \organisation{University of Science and Technology of China}
  \address{Anhui Province}
  \zip{230026}
  \city{Hefei}
  \country{P. R. China}
  \affemail{tysun@mail.ustc.edu.cn}
}

\classification[14E15, 18G80]{14F08}

\begin{abstract}
Let $X$ be a noetherian scheme with the resolution property, and let
$p:Y\to X$ be a projective morphism. Suppose that $R^ip_*=0$ for
$i>2$ and that
$
\mathcal O_X\longrightarrow Rp_*\mathcal O_Y
$
is an isomorphism. We show that derived pushforward induces an
equivalence
\[
D^b(Y)/\operatorname{Ker}(Rp_*)\simeq D^b(X).
\]
As an application, we prove a characteristic-free form of the Bondal--Orlov localization conjecture for quasi-projective threefolds.
\end{abstract}

\maketitle
\tableofcontents

\section{Introduction}\label{sec:introduction}

In 2002, Bondal and Orlov proposed the following localization statement,
now usually referred to as the Bondal--Orlov localization conjecture:

\begin{conj}[{\cite[Section~5]{BO02}, \cite[Conjecture~1.9]{Efi20}}]
\label{conj:bondal-orlov}
Let $X$ be a variety over a field $k$ of characteristic zero with rational
singularities, and let
$p:Y\longrightarrow X $
be a resolution of singularities. Then the functor
$
Rp_*:D^b(Y)\longrightarrow D^b(X)
$
between bounded derived categories of coherent sheaves is a Verdier
localization; that is, the induced functor
$
\overline{Rp_*}:
D^b(Y)/\operatorname{Ker}(Rp_*)
\longrightarrow
D^b(X)
$
is an equivalence.
\end{conj}
The conjecture is trivial in dimension one.
Bondal, Kapranov, and Schechtman proved the localization statement for
projective morphisms of quasi-projective schemes with one-dimensional
fibers
\cite[Theorem~2.14]{BKS18}. In particular, the Bondal--Orlov localization
conjecture holds for surfaces. 

Beyond the low-dimensional cases, several
special cases are known. Efimov proved a localization theorem for resolutions
obtained by blowing up a smooth center under a natural condition on the powers
of the exceptional ideal; in particular, his result applies to cones over
projectively normal embeddings of smooth Fano varieties
\cite[Theorem~1.10]{Efi20}. Pavi\'c and Shinder established the conjecture
for quotient singularities in characteristic zero
\cite[Theorem~2.30]{PS21}. More recently, Mauri and Shinder proved that
the image of
$
Rp_*:D^b(Y)\longrightarrow D^b(X)
$
generates $D^b(X)$ as a triangulated category
\cite[Theorem~1.2]{MS23}. Despite these advances, the conjecture
remains open in general.

Let $X$ be a noetherian scheme. We say that $X$ has the \emph{resolution
property} if every coherent sheaf on $X$ is a quotient of a finite-rank
locally free sheaf. In this paper, we prove a relative localization result in cohomological
dimension two. More precisely, we establish the following theorem.

\begin{theorem}
\label{thm:relative-localization}
Let $p:Y\to X$ be a projective morphism of noetherian schemes. Assume that
\begin{enumerate}
    \item $X$ has the resolution property;
    \item $R^ip_*=0$ on $\operatorname{Coh}(Y)$ for every $i>2$;
    \item the map
    $
    \mathcal O_X\longrightarrow Rp_*\mathcal O_Y
    $
    is an isomorphism.
\end{enumerate}
Then derived pushforward induces an equivalence of triangulated categories
\[
D^b(Y)\big/\operatorname{Ker}(Rp_*)
\xrightarrow{\ \sim\ }
D^b(X).
\]
\end{theorem}

As a corollary, we obtain the following characteristic-free
three-dimensional localization result.

\begin{corollary}
\label{cor:threefold}
Let $k$ be an arbitrary field, let $X$ be a quasi-projective threefold
over $k$, and let
$
p:Y\longrightarrow X
$
be a projective resolution such that
$
\mathcal O_X \xrightarrow{\sim} Rp_*\mathcal O_{Y}.
$
Then derived pushforward induces an equivalence
$
D^b(Y)/
\operatorname{Ker}(Rp_*)
\xrightarrow{\ \sim\ }
D^b(X).
$
\end{corollary}
Thus, in characteristic zero, we obtain the Bondal--Orlov localization conjecture for projective resolutions of quasi-projective threefolds with rational singularities.

\subsection{Notation}

All schemes in this paper are noetherian. For a scheme $X$, we denote by
$\operatorname{Coh}(X)$ and $\operatorname{QCoh}(X)$ the categories of
coherent and quasi-coherent sheaves on $X$, respectively. We  write
$
D^{-}(X):=D^{-}(\operatorname{Coh}(X)),\
D^{b}(X):=D^{b}(\operatorname{Coh}(X)).
$
We denote by $D_{\mathrm{std}}$ the standard $t$-structure and by
$\tau_{\mathrm{std}}^{\le m}$ and $\tau_{\mathrm{std}}^{\ge m}$ its
truncation functors. For any other $t$-structure $\mathcal P$, we use
the analogous notation $\tau_{\mathcal P}^{\le m}$ and
$\tau_{\mathcal P}^{\ge m}$. For an exact functor $F:\mathcal C\to\mathcal D$ between triangulated
categories, we write
\[
\operatorname{Ker}(F)
=
\{E\in\mathcal C\mid F(E)\simeq 0\},
\]
viewed as a full triangulated subcategory of $\mathcal C$.
\subsection{AI disclosure}

The mathematical search leading to this work was carried out with substantial assistance from AI systems. An initial proof of the Bondal--Orlov localization conjecture for threefold compound Du Val singularities, using their classification, was obtained with the assistance of Eureka, an autonomous multi-agent mathematical-reasoning system. Subsequently, GPT-5.6 Sol was used in the mathematical search that led to the more general approach presented in this paper, which avoids the classification of threefold singularities. Throughout this process, the authors guided the mathematical search through repeated prompts, selected and refined the directions pursued by the AI systems, and checked and revised the resulting arguments. The authors would like to thank the JIUCHONG team at the University of Science and Technology of China for providing complimentary access to Eureka. The authors take full responsibility for the correctness of the results and arguments presented in this paper.

\section{Preliminaries}\label{preliminaries} In this section, we recall the definition and some properties of two-tilting torsion classes.

\begin{definition}[{\cite[Definition~4.2]{Fio21}}]
Let $\mathcal A$ be an abelian category. A full subcategory
$\mathcal E\subseteq \mathcal A$ is a \emph{two-tilting torsion class} if:
\begin{enumerate}
    \item $\mathcal E$ cogenerates $\mathcal A$;
    \item $\mathcal E$ is closed under extensions in $\mathcal A$;
    \item $\mathcal E$ admits kernels internally;
    \item whenever
    \[
    0\longrightarrow A\longrightarrow E_1\longrightarrow E_2
    \longrightarrow B\longrightarrow 0
    \]
    is exact in $\mathcal A$ with $E_1,E_2\in\mathcal E$, one has
    $B\in\mathcal E$.
\end{enumerate}
Here an internal kernel is a kernel computed in the category $\mathcal E$
itself.
\end{definition}
Let $\mathcal D$ be a triangulated category. For full subcategories
$\mathcal A,\mathcal B\subset \mathcal D$, we write
\[
\mathcal A\star\mathcal B
:=
\left\{
E\in \mathcal D
\;\middle|\;
\begin{array}{l}
\text{there exists a distinguished triangle }
A\longrightarrow E\longrightarrow B\longrightarrow A[1],\\
\text{with }A\in\mathcal A,\ B\in\mathcal B
\end{array}
\right\}.
\] Fiorot proves that such a class determines a two-step tilt of the standard
$t$-structure.

\begin{proposition}[{\cite[Theorem~4.4]{Fio21}}]\label{heart of P}
Let $\mathcal E$ be a two-tilting torsion class in an abelian category
$\mathcal A$. Then
\[
\mathcal P^{\le 0}
=
D_{\mathrm{std}}^{\le -2}\star \mathcal E\star \mathcal E[1]
\]
is the aisle of a $t$-structure $\mathcal P$ on $D(\mathcal A)$, and
$
D_{\mathrm{std}}^{\le -2}
\subseteq
\mathcal P^{\le 0}
\subseteq
D_{\mathrm{std}}^{\le 0}.
$
Consequently,
$
D_{\mathrm{std}}^{\ge 0}
\subseteq
\mathcal P^{\ge 0}
\subseteq
D_{\mathrm{std}}^{\ge -2}.
$
In particular, $\mathcal P$ restricts to a $t$-structure on
$D^{-}(\mathcal A)$ and to a bounded $t$-structure on
$D^{b}(\mathcal A)$.
\end{proposition}
Let
$ p:Y\longrightarrow X $
be a projective morphism of noetherian schemes. Assume that $X$ and $p$
satisfy the assumptions of Theorem~\ref{thm:relative-localization}. Define
\[
\mathcal T
=
\left\{
T\in\operatorname{Coh}(Y):
R^ip_*T=0\text{ for }i>0,\;
p^*p_*T\twoheadrightarrow T
\right\}.
\]

\begin{proposition}
\label{prop:two-tilting}
The subcategory $\mathcal T\subseteq\operatorname{Coh}(Y)$ is a
two-tilting torsion class.
\end{proposition}

\begin{proof}
This follows from the proof of \cite[Theorem~8.7]{Fio21}. Indeed, the
argument only uses the resolution property of $X$, the projectivity of
$p$, the isomorphism
$Rp_*\mathcal O_Y\simeq\mathcal O_X, $
and the vanishing $R^ip_*=0$ for $i>2$.
\end{proof}

Applying Proposition~\ref{prop:two-tilting}, we obtain a $t$-structure $\mathcal P$ with
\begin{equation}
\label{eq:tilted-aisle}
\mathcal P^{\le 0}
=
D_{\mathrm{std}}^{\le -2}(Y)
\star\mathcal T\star\mathcal T[1].
\end{equation}
From now on, $\mathcal P$ denotes its restriction to $D^{-}(Y)$.
Its further restriction to $D^b(Y)$ is a bounded $t$-structure.

\section{The Localization Argument}

In this section, we use the same notation as in
Section~\ref{preliminaries}. We begin with the following lemma.

\begin{lemma}\label{identity}
Derived pullback and pushforward restrict to an adjunction
\[
Lp^*:D^{-}(X)\rightleftarrows D^{-}(Y):Rp_*,
\]
and the unit
$ G\longrightarrow Rp_*Lp^*G $
is an isomorphism for every $G\in D^{-}(X)$. In particular, $Lp^*$ is
fully faithful.
\end{lemma}

\begin{proof}
Since $X$ has the resolution property, every object of $D^{-}(X)$ can be
represented by a bounded-above complex of vector bundles. Hence $Lp^*$
preserves bounded-above complexes with coherent cohomology. Properness
and the finite cohomological-dimension hypothesis give the corresponding
statement for $Rp_*$. By the derived projection formula
\cite[Tag~08EU]{Stacks}, for every $G\in D^{-}(X)$ we have
\[
Rp_*Lp^*G
\simeq
G\otimes^{\mathbf L}Rp_*\mathcal O_Y
\simeq
G.
\]
Thus the adjunction unit is an isomorphism.
\end{proof}

\begin{lemma}\label{pullback}
We have
\[
Lp^*D_{\mathrm{std}}^{\le -1}(X)
\subseteq
\mathcal P^{\le -1}(Y)
\]
inside $D^{-}(Y)$, where $\mathcal P$ is the $t$-structure constructed
in Section~\ref{preliminaries}.
\end{lemma}

\begin{proof}
Let $G\in D_{\mathrm{std}}^{\le -1}(X)$. We can choose a bounded-above
vector-bundle resolution
\[
V^\bullet \simeq G
\]
with $V^i=0$ for $i\ge 0$. Note that for every $i$, the sheaf $p^*V^i$
belongs to $\mathcal T$: the projection formula and
$
Rp_*\mathcal O_Y\simeq \mathcal O_X
$
give
$
Rp_*p^*V^i\simeq V^i,
$
and the adjunction counit is an isomorphism.

Let
$
K^\bullet=[p^*V^{-2}\longrightarrow p^*V^{-1}],
$
placed in degrees $-2$ and $-1$. This is a subcomplex of
$p^*V^\bullet$, and hence there is a distinguished triangle
\begin{equation}\label{triangle}
K^\bullet\longrightarrow p^*V^\bullet
\longrightarrow p^*V^\bullet/K^\bullet
\longrightarrow K^\bullet[1].
\end{equation}
The quotient $p^*V^\bullet/K^\bullet$ is supported in degrees at most
$-3$, and hence lies in $D_{\mathrm{std}}^{\le -3}(Y)$.
Moreover, there is a distinguished triangle
\[
p^*V^{-1}[1]\longrightarrow K^\bullet
\longrightarrow p^*V^{-2}[2]\longrightarrow p^*V^{-1}[2],
\]
so that
$
K^\bullet\in\mathcal T[1]\star\mathcal T[2].
$
Shifting \eqref{eq:tilted-aisle} by $[1]$, we obtain
\[
\mathcal P^{\le -1}
=
D_{\mathrm{std}}^{\le -3}(Y)
\star\mathcal T[1]\star\mathcal T[2].
\]  Thus both $K^\bullet$ and $p^*V^\bullet/K^\bullet$ belong to
$\mathcal P^{\le-1}$. Since $\mathcal P^{\le-1}$ is extension-closed,
the  triangle~\ref{triangle} implies
$
p^*V^\bullet\in\mathcal P^{\le-1}.
$
Therefore
$
Lp^*G\simeq p^*V^\bullet\in\mathcal P^{\le-1}.
$
\end{proof}

\begin{proposition}\label{t-exact}
The functor
$
Rp_*:(D^{-}(Y),\mathcal P)
\longrightarrow
(D^{-}(X),D_{\mathrm{std}})
$
is $t$-exact.
\end{proposition}

\begin{proof}
By \eqref{eq:tilted-aisle}, every object of $\mathcal P^{\le 0}$ is
obtained by extensions from
$
D_{\mathrm{std}}^{\le -2}(Y),\
\mathcal T,\
\mathcal T[1].
$
Since $R^ip_*=0$ for $i>2$, we have
$ Rp_*\bigl(D_{\mathrm{std}}^{\le -2}(Y)\bigr)
\subseteq
D_{\mathrm{std}}^{\le 0}(X).
$
By definition of $\mathcal T$, for every $T\in\mathcal T$ the object
$Rp_*T$ is a coherent sheaf concentrated in degree zero, while
\[
Rp_*(T[1])=Rp_*T[1]\in D_{\mathrm{std}}^{\le -1}(X).
\]
Thus both $Rp_*T$ and $Rp_*(T[1])$ belong to
$D_{\mathrm{std}}^{\le 0}(X)$. Since $Rp_*$ is exact and
$D_{\mathrm{std}}^{\le 0}(X)$ is extension closed, we obtain
$ Rp_*(\mathcal P^{\le 0})
\subseteq
D_{\mathrm{std}}^{\le 0}(X).
$

For the coaisle, let $A\in\mathcal P^{\ge 0}$ and let
$G\in D_{\mathrm{std}}^{\le -1}(X)$. By
Lemma~\ref{pullback}, and the orthogonality of the
$t$-structure $\mathcal P$, we have
\[
\operatorname{Hom}_X(G,Rp_*A)
\simeq
\operatorname{Hom}_Y(Lp^*G,A)
=
0.
\]
Therefore
\[
Rp_*A
\in
\bigl(D_{\mathrm{std}}^{\le -1}(X)\bigr)^\perp
=
D_{\mathrm{std}}^{\ge 0}(X).
\]
Hence $Rp_*$ is $t$-exact.
\end{proof}

To prove the main theorem, we also need the following lemma.

\begin{lemma}\label{pushforward 0}
If $C\in D^{-}(Y)$ and $Rp_*C=0$, then for every integer $m$ the
object $\tau_{\mathcal P}^{\ge m}C$ belongs to $D^b(Y)$ and to
$\operatorname{Ker}(Rp_*)$.
\end{lemma}

\begin{proof}
By Proposition~\ref{heart of P}, we have
$
D_{\mathrm{std}}^{\ge 0}(Y)
\subseteq
\mathcal P^{\ge 0}
\subseteq
D_{\mathrm{std}}^{\ge -2}(Y).
$
After shifting, this gives
$
\mathcal P^{\ge m}
\subseteq
D_{\mathrm{std}}^{\ge m-2}(Y).
$
Since
$ \tau_{\mathcal P}^{\ge m}C\in\mathcal P^{\ge m},$
it follows that
$ \tau_{\mathcal P}^{\ge m}C
\in D_{\mathrm{std}}^{\ge m-2}(Y). $
Thus $\tau_{\mathcal P}^{\ge m}C$ is bounded below with respect to the
standard $t$-structure.

On the other hand, since
\[
\tau_{\mathcal P}^{\le m-1}C
\in
\mathcal P^{\le m-1}
\subseteq
D_{\mathrm{std}}^{\le m-1}(Y)
\]
and $C$ is bounded above, the truncation triangle
\[
\tau_{\mathcal P}^{\le m-1}C
\longrightarrow
C
\longrightarrow
\tau_{\mathcal P}^{\ge m}C
\longrightarrow
\]
shows that $\tau_{\mathcal P}^{\ge m}C$ is also bounded above. Therefore
$
\tau_{\mathcal P}^{\ge m}C\in D^b(Y).
$

Since $Rp_*$ is $t$-exact, it commutes with truncation. Hence
\[
Rp_*\tau_{\mathcal P}^{\ge m}C
\simeq
\tau_{\mathrm{std}}^{\ge m}Rp_*C
=
0.
\]
Thus
$
\tau_{\mathcal P}^{\ge m}C
\in
D^b(Y)\cap\operatorname{Ker}(Rp_*).
$
\end{proof}

\section{Proof of the Main Theorem}
We now prove Theorem~\ref{thm:relative-localization}.

\begin{proof}
The functor $Rp_*$ kills its kernel and therefore induces an exact functor
\[
\overline{Rp_*}:
D^b(Y)/\operatorname{Ker}(Rp_*)
\longrightarrow
D^b(X).
\]
Recall that every morphism between $X_1$ and $X_2$ in the Verdier quotient
$D^b(Y)/\operatorname{Ker}(Rp_*)$
can be represented by a roof
\[
X_1\xleftarrow{\,s\,}Q\xrightarrow{\,v\,}X_2,
\]
where $\operatorname{Cone}(s)\in\operatorname{Ker}(Rp_*)$. We prove that this functor is full, faithful, and essentially surjective.

\emph{Fullness.}
Let $X_1,X_2\in D^b(Y)$ and let
$ \varphi:Rp_*X_1\longrightarrow Rp_*X_2 $
be a morphism. By adjunction, $\varphi$ corresponds to a morphism
$
\widetilde{\varphi}:Lp^*Rp_*X_1\longrightarrow X_2.
$
Choose $m\ll 0$ such that
\[
X_1,X_2\in\mathcal P^{\ge m}
\qquad\text{and}\qquad
Rp_*X_1\in D_{\mathrm{std}}^{\ge m}(X).
\]
Put
\[
P:=\tau_{\mathcal P}^{\le m-1}Lp^*Rp_*X_1,
\qquad
Q:=\tau_{\mathcal P}^{\ge m}Lp^*Rp_*X_1.
\]
Then the truncation triangle gives a distinguished triangle
\begin{equation}\label{truncation}
P
\longrightarrow
Lp^*Rp_*X_1
\xrightarrow{q}
Q
\longrightarrow
P[1].
\end{equation}
Since $P\in\mathcal P^{\le m-1}$ and
$X_1,X_2\in\mathcal P^{\ge m}$, applying
$\operatorname{Hom}_Y(-,X_i)$ to \eqref{truncation} and using
orthogonality gives
\[
\operatorname{Hom}_Y(Q,X_i)
\longrightarrow
\operatorname{Hom}_Y(Lp^*Rp_*X_1,X_i)
\longrightarrow
\operatorname{Hom}_Y(P,X_i)
=0,
\qquad i=1,2.
\]
Hence both $\widetilde{\varphi}$ and the counit
$\varepsilon_{X_1}$ factor through $q$. Thus there are morphisms
$
a:Q\longrightarrow X_1,
\
b:Q\longrightarrow X_2
$
such that
$
aq=\varepsilon_{X_1},
\
bq=\widetilde{\varphi}.
$
We therefore obtain a roof
\[
X_1\xleftarrow{\,a\,}Q\xrightarrow{\,b\,}X_2.
\]
By the same argument as in Lemma~\ref{pushforward 0}, the object $Q$ is bounded. By Proposition~\ref{t-exact} and Lemma~\ref{identity},
\[
Rp_*Q
\simeq
\tau_{\mathrm{std}}^{\ge m}Rp_*Lp^*Rp_*X_1
\simeq
\tau_{\mathrm{std}}^{\ge m}Rp_*X_1
\simeq
Rp_*X_1.
\]
Thus $Rp_*q$ is an isomorphism. Since $Rp_*\varepsilon_{X_1}$ is also
an isomorphism and
$
Rp_*a\circ Rp_*q=Rp_*\varepsilon_{X_1},
$
the map $Rp_*a$ is an isomorphism. Therefore
$
\operatorname{Cone}(a)\in\operatorname{Ker}(Rp_*),
$
so $a$ is invertible in the Verdier quotient. Applying $Rp_*$ to the identities
$aq=\varepsilon_{X_1}$ and $bq=\widetilde{\varphi}$, and using that
$Rp_*q$ is an isomorphism, we obtain
\[
Rp_*b\,(Rp_*a)^{-1}
=
Rp_*\widetilde{\varphi}\,
(Rp_*\varepsilon_{X_1})^{-1}.
\]
By the adjunction identities, the right-hand side is precisely
$\varphi$. Hence the image of the roof under $\overline{Rp_*}$ is
$\varphi$, and therefore $\overline{Rp_*}$ is full.

\emph{Faithfulness.}
First let
$
u:X_1\longrightarrow X_2
$
be a morphism in $D^b(Y)$ such that $Rp_*u=0$. Complete the counit to
a distinguished triangle
\begin{equation}\label{distinguished triangle}
Lp^*Rp_*X_1
\xrightarrow{\varepsilon_{X_1}}
X_1
\longrightarrow
C
\longrightarrow.
\end{equation}
The composite
\[
Lp^*Rp_*X_1
\xrightarrow{\varepsilon_{X_1}}
X_1
\xrightarrow{u}
X_2
\]
is zero, since under adjunction it corresponds to $Rp_*u=0$.
Applying $\operatorname{Hom}_Y(-,X_2)$ to
\eqref{distinguished triangle} therefore shows that $u$ factors
through $C$. Applying $Rp_*$ to \eqref{distinguished triangle} and using
Lemma~\ref{identity} gives
$
Rp_*C=0.
$
Choose $m\ll 0$ such that $X_2\in\mathcal P^{\ge m}$. Since this map
$C\to X_2$ vanishes on $\tau_{\mathcal P}^{\le m-1}C$ by
orthogonality, it factors through
$
\tau_{\mathcal P}^{\ge m}C.
$
By Lemma~\ref{pushforward 0},
\[
\tau_{\mathcal P}^{\ge m}C
\in
D^b(Y)\cap\operatorname{Ker}(Rp_*).
\]
Thus $u$ factors through an object that is zero in the Verdier
quotient, and hence $u$ vanishes there. Now consider an arbitrary roof
\[
X_1\xleftarrow{\,s\,}Q\xrightarrow{\,v\,}X_2
\]
whose image under $\overline{Rp_*}$ is zero. Since
$
\operatorname{Cone}(s)\in\operatorname{Ker}(Rp_*),
$
the map $Rp_*s$ is an isomorphism. Hence
\[
Rp_*v\,(Rp_*s)^{-1}=0
\qquad\Longrightarrow\qquad
Rp_*v=0.
\]
By the preceding paragraph, $v$ is zero in the Verdier quotient, and
therefore the roof itself is zero. This proves faithfulness.

\emph{Essential surjectivity.}
Let $M\in D^b(X)$ and choose $m$ such that $\tau_{\mathrm{std}}^{\ge m}M
\simeq M$. Set
$
Q_M:=\tau_{\mathcal P}^{\ge m}Lp^*M.
$
The object $Q_M$ is bounded. By Proposition~\ref{t-exact} and
Lemma~\ref{identity},
\[
Rp_*Q_M
\simeq
\tau_{\mathrm{std}}^{\ge m}Rp_*Lp^*M
\simeq
\tau_{\mathrm{std}}^{\ge m}M
\simeq
M.
\]
Thus every object of $D^b(X)$ is in the essential image of
$\overline{Rp_*}$. Hence $\overline{Rp_*}$ is full, faithful, and
essentially surjective, and the theorem follows.
\end{proof}

Theorem~\ref{thm:relative-localization} immediately implies Corollary~\ref{cor:threefold}.

\begin{proof}
Since $X$ is quasi-projective, it has the resolution property
\cite[Tag~0GMM]{Stacks}. Moreover, since $p: Y\to X$ is
birational between threefolds, every fiber of $p$ has dimension at
most two. Hence
$
R^ip_*F=0
\ (i>2)
$
for every $F\in\operatorname{Coh}(Y)$. Therefore all the hypotheses of
Theorem~\ref{thm:relative-localization} are satisfied.
\end{proof}
\enlargethispage{3\baselineskip}

\bibliographystyle{emss}
\bibliography{ref}

\end{document}